\documentclass{amsart}
\usepackage[utf8]{inputenc}

\def\NZQ{\mathbb}               
\def\NN{{\NZQ N}}
\def\QQ{{\NZQ Q}}

\def\RR{{\NZQ R}}
\def\CC{{\NZQ C}}
\def\F2{{\NZQ F}_2}

\def\AA{{\NZQ A}}

\def\opn#1#2{\def#1{\operatorname{#2}}} 
\opn\chara{char} \opn\length{\ell} \opn\pd{pd} \opn\rk{rk}
\opn\projdim{proj\,dim} \opn\injdim{inj\,dim} \opn\rank{rank}
\opn\depth{depth} \opn\codepth{codepth} \opn\grade{grade}
\opn\height{height} \opn\embdim{emb\,dim} \opn\codim{codim}

\opn\Tr{Tr} \opn\bigrank{big\,rank}
\opn\superheight{superheight}\opn\lcm{lcm}
\opn\trdeg{tr\,deg}%
\opn\reg{reg} \opn\lreg{lreg} \opn\skel{skel}
\opn\Gr{Gr}
\opn\ann{ann}
\opn\sign{sign}
\opn\del{del}
\opn\lex{lex}

\opn\div{div} \opn\Div{Div} \opn\cl{cl} \opn\Cl{Cl}
\opn\Spec{Spec} \opn\Supp{Supp} \opn\supp{supp} \opn\Sing{Sing}
\opn\Ass{Ass}\opn\fdepth{fdepth}
\opn\Ann{Ann} \opn\Rad{Rad} \opn\Soc{Soc}
\opn\Sym{Sym} \opn\Ker{Ker} \opn\Coker{Coker} \opn\Im{Im}
\opn\Hom{Hom} \opn\Tor{Tor} \opn\Ext{Ext} \opn\End{End}
\opn\Aut{Aut} \opn\id{id} \opn\ini{in} \opn\tr{tr}

\opn\nat{nat}\opn\it{it}
\opn\pff{proof}
\opn\Pf{proof} \opn\GL{GL} \opn\SL{SL} \opn\mod{mod} \opn\ord{ord}
\opn\aff{aff} \opn\con{conv} \opn\relint{relint} \opn\st{st}
\opn\lk{lk} \opn\cn{cn} \opn\core{core} \opn\vol{vol}
\opn\link{link} \opn\star{star} \opn\skel{skel} \opn\indeg{indeg}
\opn\Ass{Ass} \opn\Min{Min} \opn\sdepth{sdepth} \opn\depth{depth}
\opn\gr{gr}

\def\pot#1#2{#1[\kern-0.28ex[#2]\kern-0.28ex]}

\opn\dirlim{\underrightarrow{\lim}}
\opn\inivlim{\underleftarrow{\lim}}
\def\Implies{\ifmmode\Longrightarrow \else
     \unskip${}\Longrightarrow{}$\ignorespaces\fi}
\def\implies{\ifmmode\Rightarrow \else
     \unskip${}\Rightarrow{}$\ignorespaces\fi}
\def\iff{\ifmmode\Longleftrightarrow \else
     \unskip${}\Longleftrightarrow{}$\ignorespaces\fi}

\let\:=\colon
\theoremstyle{plain}
\newtheorem{Theorem}{Theorem}[section]
 \newtheorem{Lemma}[Theorem]{Lemma}
 \newtheorem{Corollary}[Theorem]{Corollary}
 \newtheorem{Proposition}[Theorem]{Proposition}

 \theoremstyle{definition}
 \newtheorem{Definition}[Theorem]{Definition}
 \newtheorem{Remark}[Theorem]{Remark}
 
 \newtheorem{Example}[Theorem]{Example}

\let\epsilon\varepsilon
\let\kappa=\varkappa
\opn\dis{dis}
\def\pnt{{\raise0.5mm\hbox{\large\bf.}}}

\opn\Lex{Lex}

\newcommand{\rad}{1.5 pt}
\usepackage{subcaption,tikz}
\usetikzlibrary{patterns}

\usepackage{float}
\usepackage{tikz-cd}
\usepackage{mathtools}
\usepackage{amsfonts}
\usepackage{amssymb}
\newcommand{\PP}{\mathcal{P}}
\renewcommand{\AA}{\mathcal{A}}
\newcommand{\BB}{\mathcal{B}}

\newcommand{\EE}{\mathcal{E}}
\renewcommand{\CC}{\mathcal{C}}

\renewcommand{\RR}{\mathcal{R}}
\renewcommand{\SS}{\mathcal{S}}

\renewcommand{\QQ}{\mathcal{Q}}

\usepackage{enumitem}

\title{The Stanley--Reisner ideal of the rook complex of polyominoes}
\author{Francesco Romeo}
\date{July 2022}

\begin{document}

\maketitle
\begin{abstract}
    We study the properties of the rook complex $\RR$ of a polyomino $\PP$ seen as independence complex of a graph $G$, and the associated Stanley-Reisner ideal $I_\RR$. In particular, we characterise the polyominoes $\PP$ having a pure rook complex, and the ones whose Stanley-Reisner ideal has linear resolution. Furthermore, we prove that for a class of polyominoes the Castelnuovo-Mumford regularity of $I_\RR$ coincides with the induced matching number of $G$.
\end{abstract}
\section{Introduction}

Polyominoes are two-dimensional objects obtained by joining edge by edge squares of same size. Originally, polyominoes appeared in mathematical recreations \cite{Go}, but it turned out that they have applications in various fields, for example, theoretical physics and bio-informatics. Among the most popular topics in combinatorics related to polyominoes one finds enumerating polyominoes of given size, including the asymptotic growth of the numbers of polyominoes, tiling problems, and reconstruction of polyominoes. The actual research on polyominoes under an algebraic point of view focuses on the study of the polyomino ideal, a quadratic binomial ideal associated to the geometry of polyominoes (see \cite{Qu,QSS,MRR1,MRR2,CN,RR,QRR,CNU}). In the last three papers, the authors compute some algebraic invariants of the polyomino ideal by studying the rook polynomial $\sum_{i=1}^n r_it^i$, i.e.  the polynomial whose coefficient $r_i$ represents the number of distinct ways of arranging $i$ rooks on squares of $\PP$ in non-attacking positions. The degree of such polynomial is called \emph{rook number} and it is denoted by $r(\PP)$. The rook arrangements described above give rise to a simplicial complex, called \emph{rook complex}. In this paper, by focusing on the rook complex, we study polyominoes under a monomial point of view, as described below.

Let $\Delta$ be a simplicial complex on vertices $\{1,\ldots, n\}$ and let $R= K[x_1, \dots, x_{n}]$ be the polynomial ring on $n$ variables over a field $K$. The \emph{Stanley-Reisner ideal} or face ideal, denoted by $I_\Delta$, is known to be the ideal generated by the sqaure-free monomials $\{x_{i_1},\ldots x_{i_r}\}$ such that $\{i_1,\ldots,i_r\} \notin \Delta$. Let $G$ be a graph on vertices $\{1,\ldots, n\}$ and let $R= K[x_1, \dots, x_{n}]$ be the polynomial ring on $n$ variables over a field $K$. The \textit{edge ideal} of $G$, denoted by $I(G)$, is the ideal of $R$ generated by all square-free monomials $x_i x_j$ such that $\{i,j\} \in E(G)$. Edge ideals of graphs have been introduced by Villarreal \cite{Vi0} in 1990, where he studied the Cohen--Macaulay property of such ideals. Many authors have focused their attention on such ideals (e.g.\cite{HH}, \cite{CRT}). 

The two above concepts have a nice relationship. If $\Delta$ is the independence complex of $G$, i.e. the simplicial complex of the independent sets of $G$, then it holds $I(G)=I_\Delta.$ For such a reason, it is reasonable to study the Stanley-Reisner ideal of the rook complex of polyominoes. Let $\PP$ be a polyomino, let $\RR$ be its rook complex and let $I_\RR$ be the Stanley-Reisner ideal of $\RR$. 
Let $G_\PP$ be the graph on the cells of $\PP$ having $\RR$ as independence complex. It follows that $V(G_\PP)=\{C\}_{C \in \PP}$ and 
\[
E(G_\PP)=\{\{C,D\}: C \mbox{ and } D \mbox{ lie on the same row or column}\}.
\]

Some challenging problems in the modern research are the classification of Cohen-Macaulay rings and the study of the minimal free resolution and Castelnuovo-Mumford regularity. 
In Section \ref{sec:pure}, we characterise the polyominoes having a pure rook complex, i.e. all the maximal faces have the same cardinality, because the pureness is a necessary condition for the Cohen-Macaulayness. 
For the aim of studying minimal free resolution, in Section  \ref{sec:chord} we characterise the polyominoes for which $I(G_\PP)$ has linear resolution in terms of the chordality of the complement graph $\bar{G}_\PP$ (see Theorem \ref{fr}. We call such polyominoes \emph{brush} polyominoes because of their nice structure: a long interval (the handle) with dominoes as bristles. In Section \ref{sec:cmreg} we consider brush polyominoes with longer bristles and for this class we prove that the Castelnuovo-Mumford regularity of $I(G_\PP)$ coincides with the induced matching number of $G_\PP$.

\section{Preliminaries}

In this section we recall some concepts and notations on graphs and on simplicial complexes that we will use in the article. 

\subsection{Polyominoes}
In this subsection, we recall general definitions and notation on polyominoes. 

Let $a = (i, j), b = (k, \ell) \in \NN^2$, with $i	\leq k$ and $j\leq\ell$. The set $[a, b]=\{(r,s) \in \NN^2 : i\leq r \leq k \text{ and } j \leq s \leq \ell\}$ is called an \textit{interval} of $\NN^2$. Moreover, if $i<k$ and $j < \ell$, then $[a,b]$ is called a \textit{proper interval}, and the elements $a,b,c,d$ are called corners of $[a,b]$, where $c=(i,\ell)$ and $d=(k,j)$. In particular, $a,b$ are the \textit{diagonal corners} and $c,d$ are the \textit{anti-diagonal corners} of $[a,b]$. The corner $a$ (resp. $c$) is also called the left lower (resp. upper) corner of $[a,b]$, and $d$ (resp. $b$) is the right lower (resp. upper) corner of $[a,b]$. A proper interval of the form $C = [a, a + (1, 1)]$ is called a \textit{cell}. The corners of $C$ are called the vertices of $C$. The set of vertices of $C$ is denoted by $V(C)$. The edge set of $C$, denoted by $E(C)$, is 
\[
 \{\{a,a+(1,0)\}, \{a,a+(0,1)\},\{a+(1,0),a+(1,1)\},\{a+(0,1),a+(1,1)\} \}.
\]
We denote by $\ell (C)$, the left lower corner of a cell $C$.  

Let $\PP$ be a finite collection of cells of $\NN^2$, and let $C$ and $D$ be two cells of $\PP$. Then $C$ and $D$ are said to be \textit{connected}, if there is a sequence of cells $C = C_1,\ldots, C_m = D$ of $\PP$ such that $C_i\cap C_{i+1}$ is an edge of $C_i$
for $i = 1,\ldots, m - 1$. In addition, if $C_i \neq C_j$ for all $i \neq j$, then $C_1,\dots, C_m$ is called a \textit{path} (connecting $C$ and $D$). A collection of cells $\PP$ is called a \textit{polyomino} if any two cells of $\PP$ are connected. We denote by $V(\PP)=\cup _{C\in \PP} V(C)$ the vertex set of $\PP$ and by $E(\PP)=\cup _{C\in \PP} E(C)$ the edge set of $\PP$. In particular, a polyomino could be also seen as a connected bipartite graph. Note that, if $a,b\in V(\PP)$, then $a$ and $b$ are connected in $V(\PP)$ by a path of edges. More precisely,  one can find  a sequence of vertices $a=a_1,\ldots, a_{n}=b$ such that $\{a_i, a_{i+1}\} \in E(\PP)$, for all $i=1, \ldots, n-1$.  The number of cells of $\PP$ is called the \textit{rank} of $\PP$, and we denote it by $\rk \PP$. We also define the \emph{lower left corner} of $\PP$ as $\ell (\PP)= \min \{ \ell (C) : C \in \PP \}$. Each proper interval  $[(i,j),(k,l)] $ in $\NN^2$ can be identified as a polyomino and it is referred to as {\em rectangular} polyomino, or simply as rectangle. If $s=k-i$ and $t=l-j$ we say that the rectangle has {\em size} $s \times t$. In particular, given a rectangle of $\PP$ we call \emph{diagonal cells} the cells $A,B$ such that $\ell(A)=(i,j)$ and $\ell(B)=(k-1,l-1)$ and \emph{antidiagonal cells} the cells $C,D$ such that $\ell(C)=(i,l-1)$ and $\ell(D)=(k-1,j)$.

A polyomino $\PP$ is called a \emph{subpolyomino} of $\PP'$, if all cells of $\PP$ are contained in $\PP'$. Given a polyomino $\PP$, the smallest rectangle (with respect to its size) containing $\PP$ as a subpolyomino, is called the \emph{bounding box} of $\PP$.

We say that a polyomino $\PP$ is \textit{simple} if for any two cells $C$ and $D$ of $\NN^2$ not belonging to $\PP$, there exists a path $C=C_{1},\dots,C_{m}=D$ such that $C_i \notin \PP$ for any $i=1,\dots,m$. Roughly speaking, a polyomino without a ``hole'' is called a simple polyomino. We say that a polyomino $\PP$  is \emph{thin} if $\PP$ does not contain the square tetromino (see Figure \ref{fig:square}) as a subpolyomino.
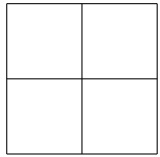
\begin{figure}[H]
\centering
\begin{tikzpicture}
\draw (0,0)--(2,0);
\draw (0,1)--(2,1);
\draw (0,2)--(2,2);

\draw (0,0)--(0,2);
\draw (1,0)--(1,2);
\draw (2,0)--(2,2);
\end{tikzpicture}\caption{The square tetromino}\label{fig:square}
\end{figure}

An interval $[a,b]$ with $a = (i,j)$ and $b = (k, \ell)$ is called a \emph{horizontal edge interval} of $\PP$ if $j =\ell$ and the sets $\{(r, j), (r+1, j)\}$ for
$r = i, \dots, k-1$ are edges of cells of $\PP$. If a horizontal edge interval of $\PP$ is not strictly contained in any other horizontal edge interval of $\PP$, then we call it \emph{maximal horizontal edge interval}. Similarly, one defines vertical edge intervals and maximal vertical edge intervals of $\PP$. 

A polyomino $\PP$ is called \emph{row convex} if for any two of its cells with lower left corners $a=(i,j)$ and $b=(k,j)$, with $k>i$, all cells with lower left corners $(l,j)$ with $i\leq l \leq k$ are cells of $\PP$. Similarly,  $\PP$ is called \emph{column convex} if for any two of its cells with lower left corners $a=(i,j)$ and $b=(i,k)$, with $k>j$, all cells with lower left corners $(i,l)$ with $j\leq l \leq k$ are cells of $\PP$. If a polyomino $\PP$ is simultaneously row and column convex then $\PP$ is called \emph{convex}. Let $\mathcal{C}:C_1, C_2, \ldots, C_m$ be a path of cells and $(i_k, j_k)$ be the lower left corner of $C_k$ for $1 \leq k \leq m$. Then $\mathcal{C}$ has a change of direction at $C_k$ for some $2 \leq k \leq m-1$ if $i_{k-1}\neq i_{k+1}$ and $j_{k-1} \neq j_{k+1}$.
A convex polyomino $\PP$ is called $k$-convex if any two cells in $\PP$ can be connected by a path of cells in $\PP$ with at most $k$ change of directions.

\subsection{Graphs and simplicial complexes}
\noindent Set $V = \{x_1, \ldots, x_n\}$. A \textit{simplicial complex} $\Delta$ on the vertex set $V$ is a collection of subsets of $V$ such that: 1) $\{x_i\} \in \Delta$ for all $x_i \in V$; 2) $F \in \Delta$ and $G\subseteq F$ imply $G \in \Delta$.
An element $F \in \Delta$ is called a \textit{face} of $\Delta$. A maximal face of $\Delta$ with respect to inclusion is called a \textit{facet} of $\Delta$.

\noindent
The dimension of a face $F \in \Delta$ is $\dim F = |F|-1$, and the dimension of $\Delta$ is the maximum of the dimensions of all facets.
Moreover, if all the facets of $\Delta$ have the same dimension, then we say that $\Delta$ is \emph{pure}. Let $d-1$ the dimension of $\Delta$ and let $f_i$ be the number of faces of $\Delta$ of dimension $i$ with the convention that $f_{-1}=1$. Then the $f$-vector of $\Delta$ is the $(d+1)$-tuple $f(\Delta)=(f_{-1},f_0,\ldots,f_{d-1})$. The $h$-vector of $\Delta$ is $h(\Delta)=(h_0,h_1,\ldots,h_d)$ with
\begin{equation} \label{eq:hf}
 h_k=\sum_{i=0}^{k}(-1)^{k-i}\binom{d-i}{k-i} f_{i-1}. 
\end{equation}
Similarly, one can express the entries of $f$-vector by the entries of the $h$-vector, in fact for $i=0,\ldots, d$
\begin{equation}\label{eq:fh}
    f_{i-1}=\sum\limits_{k=0}^i \binom{d-k}{i-k} h_k
\end{equation}

It follows that $f(\RR)=(f_{-1},\ldots, f_{d-1})$ where $f_{i-1}=r_i$ and $d=r(\PP)$. 

Let $\Delta$ be a pure independence complex of a graph $G$.
We say that $\Delta$ is \emph{vertex decomposable} if one of the following conditions hold:
(1) $n=0$ and $\Delta=\{\varnothing\}$; (2)  $\Delta$ has a unique maximal facet $\{x_{0},\ldots,x_{n-1} \}$;  (3) There exists $x \in V(G)$ such that both $\link_{\Delta}(x)$ and $\del_{\Delta}(x)$ are vertex decomposable and the facets of $\del_{\Delta}(x)$ are also facets in $\Delta$.\\
We say that $\Delta$ is \emph{Cohen-Macaulay} if for any $F \in \Delta$ we have that $\dim_K \widetilde{H}_{i}(\link_{\Delta}(F),K)=0$ for any $i < \dim \link_{\Delta}(F)$. In particular, $\Delta$ is Cohen-Macaulay if and only if $R/I_{\Delta}$ is a Cohen-Macaulay ring (see \cite{BH2}). It is well known that 
\[
\Delta \ \mbox{Vertex Decomposable} \Rightarrow \Delta \ \mbox{Cohen-Macaulay} \ \Rightarrow \Delta \ \mbox{Pure}.
\]
Let $\mathbb{F}$ be the minimal free resolution of $R/I(G)$. Then 
\[
\mathbb{F} \ : \ 0 \rightarrow F_{p} \rightarrow F_{p-1} \rightarrow \ldots \rightarrow F_{0} \rightarrow R/I(G)\rightarrow 0
\]
where $F_{i}=\bigoplus\limits_j R(-j)^{\beta_{i,j}}$. The $\beta_{i,j}$ are called the \emph{Betti numbers} of $\mathbb{F}$.
For any $i$, $\beta_{i}=\sum_{j} \beta_{i,j}$ is called the $i$-th \emph{total Betti number}.
The \emph{Castelnuovo-Mumford regularity} of $R/I(G)$, denoted by $\mbox{reg}\  R/I(G)$ is defined as 
\[
\mbox{reg} \ R/I(G)=\max\{j-i: \beta_{i,j}\neq 0 \}.
\]

Let $G$ be a graph. A collection $C$ of edges in $G$ is called an \emph{induced matching} of $G$ if the edges of $C$ are pairwise disjoint and the graph having $C$ has edge set is an induced subgraph of $G$. The maximum size of an induced matching of $G$ is called \emph{induced matching number} of $G$ and we denote it by $\nu(G)$. The \textit{complement graph} $\bar{G}$ of $G$ is the graph whose vertex set is $V(G)$ and whose edges are the non-edges of $G$. We conclude the section by stating some known results relating chordality and induced matching number to the Castelnuovo-Mumford regularity. The first one is due to Fr\"oberg (\cite[Theorem 1]{Fr})
\begin{Theorem}\label{fr}
Let $G$ be a graph. Then $\reg R/I(G) \leq 1$ if and only if $\bar{G}$ is chordal. 
\end{Theorem}
The second one is due to Katzman  (\cite[Lemma 2.2]{Ka}).
\begin{Theorem}\label{kat}
For any graph $G$, we have $\reg R/I(G)\geq \nu(G)$. 
\end{Theorem}

\section{Pureness of $\RR$}\label{sec:pure}
In this section, we characterize the polyominoes having a pure rook complex. For this aim, in the following definition, we introduce partitions on polyominoes. From now on, given two cell intervals $I$ and $J$, we write $I\cap J$ to denote the common cells of $I$ and $J$.

\begin{Definition}
Let $\PP$ be a polyomino. A subset $\varnothing \neq \AA \subset \CC$ is called a \emph{partition} of $\PP$ if 
\begin{enumerate}
    \item $\forall I,J \in \AA$ we have $I \cap J = \varnothing$;
    \item $\bigcup_{I \in A} I = \PP$.
\end{enumerate}
\end{Definition}

\begin{Example}\label{exa:part}
A polyomino admits at most two partitions, one horizontal and one vertical. In Figure \ref{fig:part}, the polyomino $\PP_1$ admits two partitions, the polyomino $\PP_2$ admits one partition and the polyomino $\PP_3$ admits no partition.
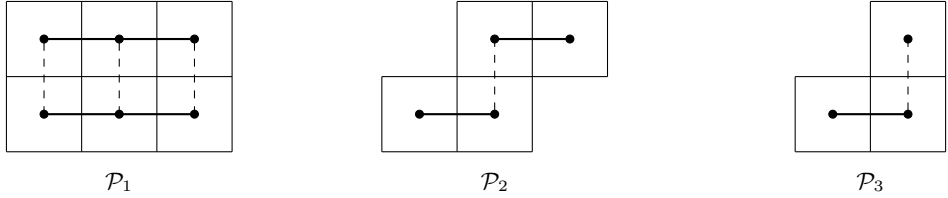
\begin{figure}[H]
    \centering 
    \begin{subfigure}{0.333\textwidth}
\centering
\begin{tikzpicture}
\draw (0,0)--(3,0);
\draw (0,1)--(3,1);
\draw (0,2)--(3,2);

\draw (0,0)--(0,2);
\draw (1,0)--(1,2);
\draw (2,0)--(2,2);
\draw (3,0)--(3,2);

\filldraw (0.5,0.5) circle (\rad);
\filldraw (1.5,0.5) circle (\rad);
\filldraw (0.5,1.5) circle (\rad);
\filldraw (1.5,1.5) circle (\rad);
\filldraw (2.5,0.5) circle (\rad);
\filldraw (2.5,1.5) circle (\rad);

\draw[thick] (0.5,0.5)--(2.5,0.5);
\draw[thick] (0.5,1.5)--(2.5,1.5);
\draw[dashed] (0.5,0.5)--(0.5,1.5);
\draw[dashed] (1.5,0.5)--(1.5,1.5);
\draw[dashed] (2.5,0.5)--(2.5,1.5);
\end{tikzpicture}
\caption*{$\PP_1$}
\end{subfigure}%
\begin{subfigure}{0.333\textwidth}
\centering
\begin{tikzpicture}
\draw (0,0)--(2,0);
\draw (0,1)--(3,1);
\draw (1,2)--(3,2);

\draw (0,0)--(0,1);
\draw (1,0)--(1,2);
\draw (2,0)--(2,2);
\draw (3,1)--(3,2);

\filldraw (0.5,0.5) circle (\rad);
\filldraw (1.5,0.5) circle (\rad);
\filldraw (2.5,1.5) circle (\rad);
\filldraw (1.5,1.5) circle (\rad);

\draw[thick] (0.5,0.5)--(1.5,0.5);
\draw[thick] (1.5,1.5)--(2.5,1.5);
\draw[dashed] (1.5,1.5)--(1.5,0.5);
\end{tikzpicture}
\caption*{$\PP_2$}
\end{subfigure}%
\begin{subfigure}{0.333\textwidth}
\centering
\begin{tikzpicture}
\draw (0,0)--(2,0);
\draw (0,1)--(2,1);
\draw (1,2)--(2,2);

\draw (0,0)--(0,1);
\draw (1,0)--(1,2);
\draw (2,0)--(2,2);

\filldraw (0.5,0.5) circle (\rad);
\filldraw (1.5,0.5) circle (\rad);
\filldraw (1.5,1.5) circle (\rad);

\draw[thick] (0.5,0.5)--(1.5,0.5);
\draw[dashed] (1.5,1.5)--(1.5,0.5);
\end{tikzpicture}
\caption*{$\PP_3$}
\end{subfigure}
\caption{We highlight with a thick line the horizontal cell intervals and with a dashed line the vertical cell intervals}\label{fig:part}
\end{figure}
\end{Example}
\begin{Definition}
A cell interval $I={C_1, C_2 \cdots C_m} \in \CC$ is called \emph{embedded} if there exists $F=\{D_1,\ldots , D_m \} \in \RR$ such that for any $i \in \{1,\ldots, m\}$ the set $\{C_i, D_i\}$ is attacking.
\end{Definition}
\begin{Remark}\label{rmk:notemb}
Let $I$ be a non-embedded interval. Then any facet $F \in \RR_\PP$ is such that $F \cap I \neq \varnothing$.
\end{Remark}
 \begin{Definition}
Let $\AA$ be a partition of $\PP$. If no interval of $\AA$ is embedded then $\AA$ is called \emph{super partition}. 
 \end{Definition}
 \begin{Example}
 We consider the polyominoes in Figure \ref{fig:part}. The partition of the polyomino $\PP_2$ and the horizontal partition of $\PP_1$ are super partitions, while the vertical one is not a super partition because any vertical cell interval is embedded. In particular referring to Figure  \ref{fig:embed}, $\{E,C\}$ embeds the interval $AD$, $\{D,C\}$ embeds the interval $BE$ and $\{D,B\}$ embeds the interval $CF$.
 
\begin{figure}[H]
\centering
\begin{tikzpicture}
\draw (0,0)--(3,0);
\draw (0,1)--(3,1);
\draw (0,2)--(3,2);

\draw (0,0)--(0,2);
\draw (1,0)--(1,2);
\draw (2,0)--(2,2);
\draw (3,0)--(3,2);
\node at (0.5,0.5) {$A$};
\node at (1.5,0.5) {$B$};
\node at (0.5,1.5)  {$D$};
\node at (1.5,1.5)  {$E$};
\node at (2.5,0.5)  {$C$};
\node at (2.5,1.5)  {$F$};
\end{tikzpicture}
\caption{All of the vertical intervals of $\PP_1$ are embedded }\label{fig:embed}
\end{figure}
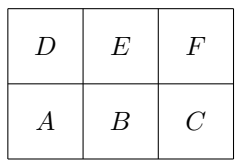
 \end{Example}
 
 We have seen that polyominoes could have at most two partitions, moreover we have the following.
 \begin{Proposition}
 A polyomino $\PP$ has two super partitions if and only if $\PP$ is a square.
 \end{Proposition}
 \begin{proof}
 If $\PP$ is the $n\times n$ square, then it has two partitions given by the $n$ rows and the $n$ columns. Moreover, any interval can be attacked at most on $n-1$ cells, that is, the two partitions have no embedded intervals. 
 
 Conversely, let $\AA$ and $\BB$ be two super partitions of $\PP$. Let $I_1=C_1C_2\cdots C_n \in \AA$ with a free edge interval, wlog the uppermost one. Let $J_1,\ldots, J_n \in \BB$ be such that $I_1 \cap J_i=\{C_i\}$. Let $m= \min_{i}\{|J_i|\}$, $t \in \{1,\ldots, n\}$ be such that $|J_t|=m$ and let $I_1, \ldots I_m \in \AA$ be the rows covering the intervals $J_1,\ldots J_n$. We claim $m=n$. \\
 If $m<n$, for any subset $\{i_1,\ldots,i_m\}\subset {1,\ldots, \hat{t}, \ldots,n}$ it holds that $F= \{D_1,\ldots, D_m\} \in \RR$, where $D_k=I_{k}\cap J_{i_k}$. That is $J_t$ is embedded and $\BB$ is not a super partition of $\PP$.\\
 If $m>n$,  for any $k \in \{1,\ldots n\}$ let $D_{k}=I_{k+1} \cap J_{k}$, then $F=\{D_1,\ldots, D_n\}$ attacks all the cells of $I_1$.That is $I_1$ is embedded and $\AA$ is not a super partition of $\PP$.
 That is $m=l$.\\
 Moreover, if there exists a $k$ such that  $J_{k}=D_1 D_2 \cdots D_{s}$ with $s > m$, then let $F$ be a configuration on the $n\times nn$ square containing $C_k$. The configuration $F'= F \setminus \{C_k\} \cup D_{m+1}$ is non-attacking and covers $I_1$, that is $\AA$ is not a super partition of $\PP$. 
 The latter shows that $I_1,\ldots, I_n$ are the only intervals of the partition $\AA$. By a similar argument, $J_1, J_2, \ldots J_n$ are the only intervals of $\BB$, that is $\PP$ is the $n \times n$ square.
 \end{proof}
 
 \begin{Proposition}
 Let $\PP$ be a non-square polyomino with unique super partition $\AA$. Then for any $J \in \CC\setminus \AA$, $J$ is embedded.
 \end{Proposition}
 \begin{proof}
 Given $J \in \CC\setminus \AA$ with $J=B_1 B_2 \cdots B_m$, one can focus on the polyomino $\PP=\PP_J$ given by $I_1,\ldots, I_m \in \AA$ such that for any $k \in 1,\ldots ,m$ $I_k \cap J= B_k$.
 We proceed by induction on $m$.
 Let $m=2$ and assume that $J=B_1 B_2$ is not embedded. That is, if $D_1 \in I_1$ attacks $B_1$ and $D_2 \in I_2$ attacks $B_2$, $D_1$ and $D_2$ are attacking, hence there are no other cells in $\PP$ except $B_1,B_2,D_1,D_2$ and $\PP$ is the $2 \times 2$ square. Contradiction. 
Let $m > 2$. We assume that for any polyomino $\PP'$ with super partition $\BB$, any interval $J'$ in $\CC_\PP' \setminus \BB$ with $|J'|<m$ is embedded. 
We consider $I_1$. By the definition of polyomino $\PP_J$, $I_1$ has a free edge interval, wlog assume the super partition is made by horizontal intervals and the free edge is the uppermost one, and it contains the cell $B_1$ of $J$.
Assume $I_1=C_1 C_2 \ldots C_n$ and $C_t=B_1$ for $1< t \leq n$ (if $B_1=C_1$, apply the following arguments to the cell $C_n$). Let $J_1=D_1 D_2 \ldots D_l$ with $D_1=C_1$ and by construction $l \leq m$. We consider the polyomino $\PP'$ given by $\PP \setminus (I_1 \cup J_1)$. We divide two cases: 
\begin{enumerate}
    \item $\PP'$ is a square;
    \item $\PP'$ is not a square.
\end{enumerate}

In case (1), since $\PP'$ contains $J\setminus B_1$, then $\PP'$ is a $(m-1) \times (m-1)$ square and assume its cells are $\{A_{ij}\}_{i,j \in\{2,\ldots,m\}}$.
We prove that the length $l$ of $J_1$ is equal to $m$ and that $C_2$ and $A_{22}$ are on the same column, by proving that $I_m=D_m A_{m 2} A_{m 3}\cdots A_{mm}$.
By contraposition, assume $I_m= A_{m 2} A_{m 3} \cdots A_{mm}$ and a set $F\in \RR_{\PP'}$ with $|F|=m-1$ containing the cell $B_m \in I_m$. Then $F'=(F \setminus \{B_m\}) \cup \{B_1\} \in \RR_\PP$ since $B_1 \in I_1$. We have that $F'$ attacks the interval $I_m \in \AA$, that is $\AA$ is not a super partition. This leads to a contradiction. In particular from $ A_{m 2} A_{m 3}\cdots A_{mm} \subset I_m$ and $I_m \setminus J_1=A_{m 2} A_{m 3}\cdots A_{mm}$, we obtain $I_m=D_m A_{m 2} A_{m 3}\cdots A_{mm}$.
Moreover, by similar arguments, if the length $n$ of $I_1$ is less than $m$, one can find an embedding for $I_1$, given by the set $\{D_2, A_{32},\ldots ,A_{n+1 n}\}$. That is $n \geq m$. If $n=m$, then $\PP$ is the $m\times m$ square, that contradicts the hypothesis, hence $n > m$. That is, the $m \times m$ square is a subpolyomino of $\PP$. Therefore, if $F \in \RR_\PP$ is a set of $m$ non-attacking rooks on such a square containing $B_1$, that is $F\setminus {B_1}$ embeds $J \setminus {B_1}$, then $F'= (F\setminus \{B_1\} \cup \{C_n\})$ embeds the interval $J$.

In case (2), the set  $\AA'=\{ I_2 \setminus \{D_2\},  \ldots , I_l \setminus \{D_l\}, I_{l+1}, \ldots I_m\}$ is a partition of $\PP'$. We prove that $\AA'$ is a super partition. If one of the intervals $I_{l+1}, \ldots I_m$ is embedded, then it is embedded in $\AA$, contradiction. If an interval $I_k \setminus D_k$ is embedded by a configuration $F' \in \RR_{PP'}$, then $D_1 \cup F'$ embeds $I_k$ in $\PP$, and this is a contradiction. That is $\AA'$ is a super partition of $\PP'$.
Moreover, the interval $J \ \{B_1\}$ has cardinality $m-1$ and by induction hypothesis is embedded by a configuration $F' \in \RR_{\PP'}$, and $F' \cup \{D_1\}$
is an embedding for $J$ as desired.
\end{proof}
\begin{Lemma}\label{lem:embneq}
 Let $\PP$ be a polyomino and let $I,I' \in \CC$ be such that $I$ is embedded and there exists $J \in \CC$ with $J\cap I \neq \varnothing$ and $J\cap I \neq \varnothing$. Then there exists an embedding $F$ of $I$ such that $F \cap I' \neq \varnothing$.
\end{Lemma}
\begin{proof}
Let $G=\{D_1,D_2, \ldots D_l\}$ be an embedding of $I=C_1C_2 \cdots C_l$ and assume $G \cap I'= \varnothing$. Assume that $J\cap I=C_j$ and $J\cap I'=D$. Then we claim $F=G\setminus \{D_j\} \cup \{D\}\in \RR_\PP$. If $D$ is attacked by a $D_k \in G$, then $D_k \in I'$, that is $G \cap I'\neq  \varnothing$.
Hence, $F$ is an embedding for $I$ with $F\cap I'\neq \varnothing$ as desired.
\end{proof}
 We now prove the main theorem of this section
 \begin{Theorem}\label{thm:main}
 Let $\PP$ be a polyomino. The following are equivalent:
 \begin{itemize}
     \item[(1)] $\RR_\PP$ is pure and has dimension $d-1$;
     \item[(2)] $\PP$ admits a super partition with $|\AA|=d$.
 \end{itemize}
 \end{Theorem}
 \begin{proof}
 $(2) \Rightarrow (1)$. Assume $\PP$ has a super partition $\AA$ with $|\AA|=d$. 
 By contraposition, assume that $\RR_\PP$ is not pure, that is there exists a facet $F$ with $|F|=t < d$. Let $I_1,\ldots I_t \in \AA$ be the intervals containing the $t$ cells of $F$. Since $t < d $, then $\AA \setminus \{ I_1,\ldots I_t \} \neq \varnothing$, that is, there exists $I \in \AA \setminus \{ I_1,\ldots I_t \}$ that is embedded by $F$. Hence, $\AA$ is not a super partition. Contradiction.
 
  $(1) \Rightarrow (2)$. Assume that $\RR_\PP$ is pure and has dimension $d-1$. We divide the proof in the following steps:
 \begin{itemize}
     \item We prove that $\PP$ admits a partition $\AA$;
     \item If $\EE \subseteq \AA$ is the set of the embedded intervals of $\AA$, then either $\EE = \varnothing$ or $\EE= \AA$.
     \item In the case $\AA=\EE$, $\PP$ admits a super partition $\BB$.
\end{itemize}

  \textit{Existence of $\AA$.} By contraposition, assume that $\PP$ admits no partition, that is there exists a vertical interval $I$ with a single cell $C$ and a horizontal interval with a single cell $D$. 
  Since $\PP$ is a polyomino, the cells $C$ and $D$ are $k$-connected with $k$ odd and let $C_1, C_2, \ldots, C_k$ be the changes of direction. We consider the set $F=\{C=C_0, C_2, C_4 \ldots, C_{k-1},D=C_{k+1}\}$ that lies in $\RR_\PP$, because for any $i \in \{2,4,\ldots, k-1\}$ $C_i$ only attacks $C_{i-1}$ and $C_{i+1}$. Since $\RR_\PP$ is pure, then there exists $G \in \RR_\PP$ such that $|F \cup G|=d$. We now consider $F'=\{C_1, C_3, \ldots, C_k\} \in \RR_\PP$ and we claim that $F' \cup  G \in \RR_\PP$. If $A \in G$ attacks $C_i$ with $i=1,3, \ldots k$, then $A$ attacks either $C_{i-1} \in F$ or $C_{i+1} \in F$, that is $F\cup G \notin \RR_\PP$. Contradiction. \\
  We prove that $F' \cup G$ is maximal, i.e. any cell $A$ of $\PP$ is attacked by a cell of $F' \cup G$. 
  Since $F \cup G$ is maximal, then for any $A \in \PP$ there exists $B \in F \cup G$ attacking $A$. If $B \in G$, then $B \in F' \cup G$.  If $B \in F$, then $B= C_{i}$ for $i \in {0,2,\ldots, k+1}$ and either $C_{i-1}$ or $C_{i+1}$ in $F'$ attacks $B$, and hence $A$. We showed that $F' \cup G$ with $|F' \cup G|< |F \cup G|$ is maximal, hence $\RR_\PP$ is not pure. Contradiction. Hence, $\PP$ admits a partition $\AA$ and since $\dim \RR_\PP=d-1$, then $|\AA|\geq d$. 
  
  \textit{The set $ \EE $ of embedded intervals.} Let $\varnothing \subseteq \EE \subseteq \AA$ be the set of embedded intervals of $\AA$. We claim that either $\EE = \varnothing$ or $\EE= \AA$. If $\varnothing \neq \EE \neq \AA$, let $I \in \AA\setminus \EE$ be a non-embedded interval and let $I' \in \EE$, and since $\AA$ is a partition, there exists $C \in I$ and $C' \in I'$ such that $C$ and $C'$ are $2k$-connected with $k \geq 1$ with intervals $I_1=I \ldots I_k=I' \in \AA$. Wlog one can assume that $k=1$, in fact if $I_2$ is embedded then take $I'=I_2$, otherwise take $I=I_2$. That is, assume $I$ and $I'$ are $2$-connected, that is, there exists $J\in \CC$ such that $J\cap I\neq \varnothing$ and $J\cap I'\neq \varnothing$. From Lemma \ref{lem:embneq}, there exists an embedding $F$ of $I'$ such that $F \cap I \neq \varnothing$. Let $G$ be a facet of $\RR_\PP$ containing $F$ and let $G'= G \setminus \{A\} \cup C_j$. We observe that $G' \cap I = \varnothing$ that is either $I$ is embedded or by Remark \ref{rmk:notemb}, there exists a facet $\bar G$ with $\bar G \cap I \neq \varnothing$ containing $G'$, that are both contradictions. 
  Hence either $\EE = \varnothing$ or $\EE= \AA$. In the case $\AA=\varnothing$, we have that $\AA$ is a super partition and $|\AA|=d$ due to Remark \ref{rmk:notemb}. 
  
  \textit{The case $\EE=\AA$}. Any interval of $\AA$ is embedded, that is $|\AA| > d$ and, in particular, no interval of $\AA$ has single cells, therefore $\PP$ admits another partition $\BB$. Assume that $\AA$ contains rows and $\BB$ contains columns. We claim that $\BB$ contains no embedded intervals. By contraposition, assume that $J=C_1 C_2 \cdots C_l \in \BB$ is embedded by $F=\{D_1,\ldots, D_l\}$ and let $I_1,\ldots I_l$ be its rows. Since $l\leq d$ and $|\AA|>d$, then there exists $I \in \AA$ that is embedded by a facet $G $ containing $ F$. That is, a cell $C_j$ of $J$ is $2k$-connected to a cell $D$ of $I$ with $k\geq 1$ by a path with columns given by $J_1,J_2 \ldots J_{k}$ and changes of directions $L_1,L_2\ldots L_{2k}$. We may assume that $k=1$. In fact, if one of the rows of $J_1$ is embedded, then $C_j$ is $2$-connected to a cell $D'$ of an embedded interval. Otherwise, from Lemma \ref{lem:embneq}, one can choose $F$ such that $J_1 \cap F = \{D_k\}$. Since $G$ is maximal and $J_1$ has no embedded rows, any cell of $J_1 \setminus \{D_k\}$ is attacked by a cell of $G$, that is $G \setminus \{D_k\} \cup \{C_k\}$ is an embedding of $J_1$ and the cell $L_2 \in J_1$ is $2(k-1)$-connected to the cell $D$ of $I$. That is, we assume $C_j \in J$ and $D \in I$ are $2$-connected and we prove that there exists $\bar F$ embedding of $I$ such that $F \cap \bar{F}\neq \varnothing$. For this aim, let $F'$ be an embedding for $I$ and assume $F' \cap F = \varnothing$. Let $J'$ be such that $J'\cap I_j \neq \varnothing$ and $J' \cap I \neq \varnothing$. From Lemma \ref{lem:embneq} applied to $J,J'$ and $I_j$,we can choose $F$ such that $F \cap J' =\{D_j\}$. Moreover, since $I$ is embedded, then $J' \cap F'=\{D'\}$. Take $\bar F= F'\setminus \{D'\} \cup \{D_j\}$, that is $F \cap \bar{F}\neq \varnothing$.
  Let $G'$ be a maximal face containing $\bar F \cup F$, then a cell $A \in F \cap \bar{F}$ attacks a cell $C \in J$ and $D \in I$, that is the face $G' \setminus \{A\} \cup \{C,D\}$ is maximal and $|G' \setminus \{A\} \cup \{C,D\}|>G'$, contradiction to the pureness of $\RR_\PP$. This shows that $\BB$ contains no embedded intervals, hence $|\BB|= d$ and it is a super partition of $\PP$ and by Remark \ref{rmk:notemb} one has  $|\BB|=d$.
 \end{proof}

    We observe that a partition on $\PP$ induces a clique partition in the associated graph $G_\PP$. Moreover, if such partition is \emph{super} then the graph $G_\PP$ is said to be \emph{localizable} (see \cite{HMR}). For a graph $G$, if $\Delta(G)$ is pure, then we say that $G$ is \emph{well-covered}. Hence, Theorem \ref{thm:main} can be rephrased as follows:
    \begin{Theorem}
        Let $G_\PP$ be a graph associated to a polyomino $\PP$. Then the following are equivalent:
       \begin{enumerate}
           \item $G_\PP$ is well-covered;
           \item $G_\PP$ is localizable.
       \end{enumerate}
    \end{Theorem}

 \section{Chordality of $\bar G_\PP$}\label{sec:chord}
 In this section we characterize the polyominoes $\PP$ for which the graph $\bar{G}$ is chordal. In view of Theorem \ref{fr}, we obtain information on the minimal free resolution of $I_\RR$. We start with the following Lemma.
 \begin{Lemma}
 Let $\PP$ be a polyomino and let $\gamma=\{A_1,\ldots, A_n\}$ be an induced cycle of $\bar G$. Then $n \in \{3,4,6\}$.
 \end{Lemma}
 \begin{proof}
 We assume $n >6$. Since $\gamma$ is induced, then $A_3, A_4, A_5,A_6 \in N_{G}(A_1) $. Moreover, $\{A_3,A_5\}$, $\{A_4,A_6\}\in E(G)$, that is there exists a cell interval $I$ of $\PP$ containing $A_1,A_3,A_5$ and a cell interval $J$ of $\PP$ containing $A_1,A_4,A_6$, as shown in Figure \ref{fig:cross}. We have that $\{A_3,A_6\} \in E(\bar G)$ and $\gamma$ is not induced. Hence an induced cycle in $\bar{G}$ has length less than or equal to $6$. \\

 \begin{figure}[H]
     \centering
     \begin{tikzpicture}
      \draw (1,0)--(2,0);
      \draw (0,1)--(3,1);
      \draw (0,2)--(3,2);
       \draw (1,3)--(2,3);

      \draw (0,1)--(0,2);
      \draw (1,0)--(1,3);
      \draw (2,0)--(2,3);
       \draw (3,1)--(3,2);
       
       \node at (1.5,0.5){$A_6$};
       \node at (1.5,1.5){$A_1$};
       \node at (0.5,1.5){$A_3$};
       \node at (1.5,2.5){$A_4$};
       \node at (2.5,1.5){$A_5$};
     \end{tikzpicture}
     \caption{}\label{fig:cross}
 \end{figure}
  If $n=5$, then $\{A_1,A_3\}, \{A_1,A_4\} \in E(G)$ and $\{A_3, A_4\}\notin E(G)$, that is there exists a cell interval $I$ of $\PP$ containing $A_1$ and $A_3$ and a cell interval $J$ containing $A_1$ and $A_4$. By similar arguments, there exists a maximal cell interval $I' \neq J$ containing $A_2$ and $A_4$ and $A_3$ and $A_2$ are $1$-connected with change of direction at $A_5$. This implies that either $\{A_1,A_5\} or \{A_4,A_5\}\in E(G)$ and  $\gamma$ is not a cycle. This proves that $n\neq 6$ and $n \neq 5$ as desired.
 \end{proof}
 \begin{Remark}\label{rem:hex}
We highlight that a cycle $\gamma=\{A_1,\ldots , A_6\}$ of length $6$ is given by the hexomino in Figure \ref{fig:hexomino}. That is, whenever the above hexomino is subpolyomino of $\PP$ we have that $\bar G$ is not chordal.
 \begin{figure}[H]
     \centering
     \begin{tikzpicture}
      \draw (0,0)--(2,0);
      \draw (0,1)--(2,1);
      \draw (0,2)--(2,2);
       \draw (0,3)--(2,3);

      \draw (0,0)--(0,3);
      \draw (1,0)--(1,3);
      \draw (2,0)--(2,3);

       \node at (0.5,0.5){$A_1$};
       \node at (1.5,0.5){$A_4$};
       \node at (0.5,1.5){$A_3$};
       \node at (1.5,1.5){$A_6$};
       \node at (0.5,2.5){$A_5$};
       \node at (1.5,2.5){$A_2$};
     \end{tikzpicture}
     \caption{}\label{fig:hexomino}
 \end{figure}
\end{Remark}

We prove some results under the assumption that $\bar G$ is chordal.
\begin{Remark}\label{rem:kconv}
If $\bar G$ is chordal, then any two cells in $\PP$ are $k$-connected with $k < 3$. In fact, assume that $C,D \in \PP$ are $k$-connected with $C_1,\ldots, C_k$ changes of directions and $k>3$, that is $C_1$ and $C_k$ lie on different cell intervals, hence $\{C,C_k, C_1, D\}$ is an induced $4$-cycle of $\bar G$.
\end{Remark}

\begin{Proposition}\label{prop:geq2}
Let $\bar G$ be a chordal graph and assume that there exists $J \in \CC$ such that $|J|>2$. Then for any $I \in \CC \setminus \{J\}$ one has $|I|=2$.
\end{Proposition}
\begin{proof}
By contraposition, assume that there exists $I \in \CC \setminus \{J\}$ such that $|I|>2$. We distinguish two cases:
\begin{enumerate}
    \item $I\cap J \neq \varnothing$;
    \item $I\cap J = \varnothing$.
\end{enumerate}
In case $(1)$, let $I\cap J= \{C\}$. It follows that 
$|I\setminus \{C\}| \geq 2$ and $|J\setminus \{C\}| \geq 2$, that is there exist $C_1,C_2 \in I$ and $D_1, D_2 \in J$ such that $\{C_1,D_1,C_2,D_2\}$ is an induced $4$-cycle of $\bar{G}$, that is $\bar{G}$ is not chordal.  \\
In case $(2)$, from Remark \ref{rem:kconv}, a cell of $I$ is at most $2$-convex with a cell of $J$, that is let $I_1,\ldots, I_l$ be intervals such that for any $j \in \{1,\ldots, l\}$ $I_j \cap I\neq \varnothing$ and $I_j \cap J\neq \varnothing$. If $l>2$, then the hexomino of Figure \ref{fig:hexomino} is a subpolyomino of $\PP$ and from Remark \ref{rem:hex} $\bar G$ is not chordal. That is either $l=1$ or $l=2$. In the former case there exist $C_1,C_2 \in I\setminus (I_1 \cap I) $ and $D_1, D_2 \in J\setminus (I_1 \cap J)$ such that $\{C_1,D_1,C_2,D_2\}$ is an induced $4$-cycle of $\bar{G}$, that is $\bar{G}$ is not chordal, while in the latter case there exist $C_1,C_2 \in I\setminus (I_1 \cap I) $ and $D_1, D_2 \in J\setminus (I_2 \cap J)$ such that $\{C_1,D_1,C_2,D_2\}$ is an induced $4$-cycle of $\bar{G}$, that is $\bar{G}$ is not chordal.
\end{proof}
 \begin{Corollary}
 Let $\PP$ be a polyomino with a chordal $\bar G$, then $\PP$ is simple.
  \end{Corollary}
  \begin{proof}

  In fact, if $\PP$ is non-simple, then it is multiply connected, that is, there exists a closed path of cells. If such closed path has more than $4$ intervals, then there are two cells that are $k$-connected with $k>3$, that is $\bar{G}$ is not chordal due to Remark \ref{rem:kconv}. That is there are $4$ intervals $I_1,\ldots I_4$. Since $\PP$ is non simple then we should have $|I_1|>2$ and $|I_2|>2$, that contradicts Proposition \ref{prop:geq2}.  
  
  \end{proof} 
 For the aim of classifying the polyominoes having a chordal $\bar G$, we first characterize the simple non-thin ones.
\begin{Lemma}
Let $\PP$ be a simple non-thin polyomino. Then $\bar{G}$ is chordal if and only if $\PP$ is one of the two polyominoes in Figure \ref{fig:nnthin}.
\begin{figure}[H]
    \centering
    \begin{subfigure}{0.45 \textwidth}
      \centering
    \begin{tikzpicture}
     \draw (0,0)--(2,0);
\draw (0,1)--(2,1);
\draw (0,2)--(2,2);

\draw (0,0)--(0,2);
\draw (1,0)--(1,2);
\draw (2,0)--(2,2);

    \end{tikzpicture}
    \end{subfigure}\hfill%
    \begin{subfigure}{0.45 \textwidth}
      \centering
        \begin{tikzpicture}
     \draw (0,0)--(3,0);
\draw (0,1)--(3,1);
\draw (0,2)--(2,2);

\draw (0,0)--(0,2);
\draw (1,0)--(1,2);
\draw (2,0)--(2,2);
\draw (3,0)--(3,1);
    \end{tikzpicture}
    \end{subfigure}
    \caption{}\label{fig:nnthin}
\end{figure}
\end{Lemma}
\begin{proof}
The graphs $\bar{G}$ associated to the polyominoes of Figure \ref{fig:nnthin} are the graphs in Figure \ref{fig:cho} that are clearly chordal.
\begin{figure}[H]
    \centering
    \begin{subfigure}{0.45 \textwidth}
      \centering
    \begin{tikzpicture}
\filldraw (0,0) circle (\rad);
\filldraw (1,1) circle (\rad);
\filldraw (0,1) circle (\rad);
\filldraw (1,0) circle (\rad);
 \draw (0,0)--(1,1);
 \draw (0,1)--(1,0);
    \end{tikzpicture}
    \end{subfigure}\hfill%
    \begin{subfigure}{0.45 \textwidth}
      \centering
        \begin{tikzpicture}
\filldraw (0,0) circle (\rad);
\filldraw (1,1) circle (\rad);
\filldraw (0,1) circle (\rad);
\filldraw (1,0) circle (\rad);
\filldraw (2,0) circle (\rad);
 \draw (0,0)--(1,1);
 \draw (0,1)--(1,0);
  \draw (2,0)--(1,1);
 \draw (0,1)--(2,0);
    \end{tikzpicture}
    \end{subfigure}
    \caption{}\label{fig:cho}
\end{figure}

We now assume that $\bar G$ is chordal and $\PP$ strictly contains the square tetromino $\QQ=\{C_1,C_2,C_3,C_4\}$ with $\{C_1,C_2\}$ non-attacking and let $C \in \PP \setminus \QQ$. Let $I$ be the cell interval containing $C$. If there exists $D \in I$ such that $D \notin \{C_1,C_2,C_3,C_4\}$, then either $\{C,D,C_1,C_4\}$ or  $\{C,D,C_1,C_3\}$ or $\{C,D,C_2,C_3\}$  or $\{C,D,C_2,C_4\}$ is an induced $4$-cycle of $\bar G$. That is, if $I\setminus {C} \subseteq \{C_1,C_2,C_3,C_4\}$, assume $I=\{C,C_2,C_3\}$, as in Figure \ref{fig:pent}.
\begin{figure}[H]
    \centering
        \begin{tikzpicture}
     \draw (0,0)--(3,0);
\draw (0,1)--(3,1);
\draw (0,2)--(2,2);

\draw (0,0)--(0,2);
\draw (1,0)--(1,2);
\draw (2,0)--(2,2);
\draw (3,0)--(3,1);
  \node at (0.5,1.5){$C_1$};
       \node at (1.5,0.5){$C_2$};
       \node at (0.5,0.5){$C_3$};
       \node at (1.5,1.5){$C_4$};
       \node at (2.5,0.5){$C$};
   
    \end{tikzpicture}
    \caption{}\label{fig:pent}
\end{figure}
We assume that there exists $D \in \PP \setminus \{C_1,C_2,C_3,C_4,C\}$ and let $J \in \CC$ be such that $D \in J$. It follows that either $|J|>3$ or $J=\{C,D\}$ and $J \cap \{C_1,C_4\}= \varnothing$. In the former case, both $I$ and $J$ have cardinality greater than $3$, contradicting Proposition \ref{prop:geq2}. In the latter case, $C_1$ and $D$ are $3$-connected contradicting Remark \ref{rem:kconv}.
That is, if $\bar{G}$ is chordal $\PP$ is one of the two polyominoes in Figure \ref{fig:nnthin}.
\end{proof}

We are left with the characterization of the simple thin polyominoes having a chordal $\bar{G}$.  We introduce a new class of polyominoes.
\begin{Definition}
A simple thin polyomino $\PP$ such that $\CC=\{J,I_1,\ldots, I_l\}$ with $0\leq l \leq |J|$ and for any $k \in \{1,\ldots , l\}$ $I_k\cap J \neq \varnothing$ is called a \emph{brush polyomino}(see Figure \ref{fig:brush}). If in addition for any $k \in \{1,\ldots , l\}$ $|I_k|=2$, then is $\PP$ is called a \emph{short brush polyomino}(see Figure \ref{fig:sbrush}).
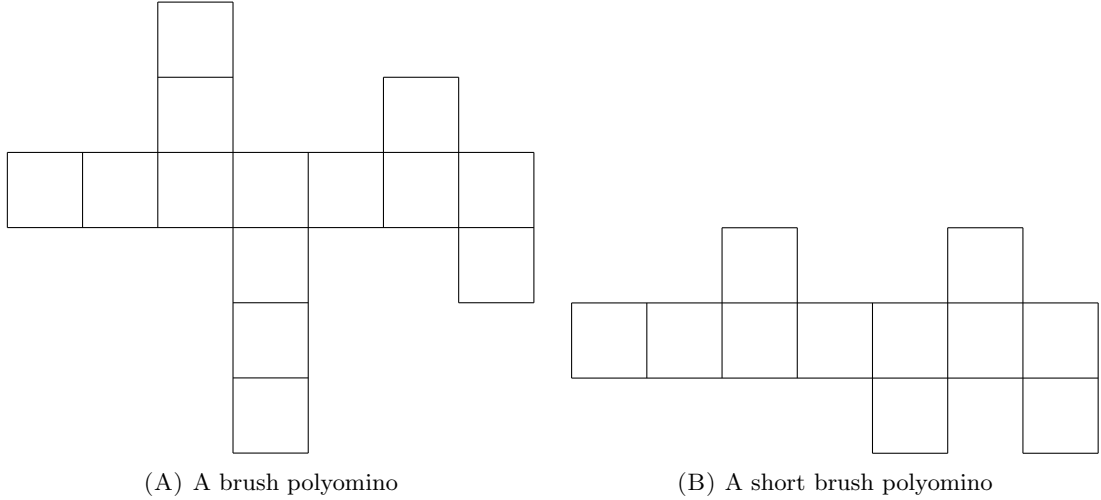
\begin{figure}[H]
\centering
\begin{subfigure}[t]{0.5 \textwidth}
    \centering
        \begin{tikzpicture}
\draw (-1,0)--(6,0);
\draw (-1,1)--(6,1);

\draw (1,2)--(2,2);
\draw (1,3)--(2,3);
\draw (2,-1)--(3,-1);
\draw (2,-2)--(3,-2);
\draw (2,-3)--(3,-3);
\draw (5,2)--(4,2);
\draw (5,-1)--(6,-1);

\draw (-1,0)--(-1,1);
\draw (0,0)--(0,1);
\draw (1,0)--(1,3);
\draw (2,-3)--(2,3);
\draw (3,-3)--(3,1);
\draw (4,0)--(4,2);
\draw (5,-1)--(5,2);
\draw (6,-1)--(6,1);

    \end{tikzpicture}
    \caption{A brush polyomino}\label{fig:brush}
    \end{subfigure}%
\begin{subfigure}[t]{0.5 \textwidth}
    \centering
        \begin{tikzpicture}
\draw (-1,0)--(6,0);
\draw (-1,1)--(6,1);

\draw (1,2)--(2,2);
\draw (3,-1)--(4,-1);
\draw (5,2)--(4,2);
\draw (5,-1)--(6,-1);

\draw (-1,0)--(-1,1);
\draw (0,0)--(0,1);
\draw (1,0)--(1,2);
\draw (2,0)--(2,2);
\draw (3,-1)--(3,1);
\draw (4,-1)--(4,2);
\draw (5,-1)--(5,2);
\draw (6,-1)--(6,1);

    \end{tikzpicture}
    \caption{A short brush polyomino}\label{fig:sbrush}
    \end{subfigure}
\end{figure}
\end{Definition}

\begin{Theorem}
Let $\PP$ be a simple thin polyomino. Then $\bar{G}$ is chordal if and only if $\PP$ is a short brush polyomino.
 \end{Theorem}
\begin{proof}

Assume $\PP$ is a brush polyomino and that $\{A_1,A_2,A_3,A_4\}$ is an induced cycle of $\bar{G}$ with $A_1$ and $A_3$ on a cell interval, and $A_2$ and $A_4$ on another cell interval. From the structure of $\PP$ it follows that one between $A_1$ and $A_3$ lies on $J$, say $A_3$. Similarly, wlog assume $A_2 \in J$, that is $A_2,A_3 \in E(G)$ and $\bar{G}$ is not a cycle.\\
We now assume that $\PP$ is a simple thin polyomino with a chordal $\bar G$. From Remark \ref{rem:kconv}, we have that $\PP$ is $k$-connected with $k \leq 2$. We now distinguish two cases.
\begin{enumerate}
    \item for any $I \in \CC$ it holds $|I|=2$;
    \item there exists $J \in \CC$ such that $|J|>2$. 
\end{enumerate}

In case (1), we have that $|\CC|\leq 3$, in fact the only polyominoes satisfying the above property are subpolyominoes of the skew tetromino in Figure \ref{fig:skew}.  In fact, if $\CC$ contains another cell interval $I$ with $|I|=2$, then $I$ contains $C$ (resp. $D$) and another cell $A$ that is $3$-connected to $D$ (resp. $C$).
\begin{figure}[H]
    \centering
            \begin{tikzpicture}
\draw (0,0)--(2,0);
\draw (0,1)--(3,1);
\draw (1,2)--(3,2);

\draw (0,0)--(0,1);
\draw (1,0)--(1,2);
\draw (2,0)--(2,2);
\draw (3,1)--(3,2);

\node at (0.5,0.5){$C$};
\node at (2.5,1.5){$D$};
    \end{tikzpicture}
    \caption{The skew tetromino}
    \label{fig:skew}
\end{figure}
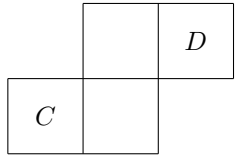
In case $(2)$, from Proposition \ref{prop:geq2}, we obtain that $\forall I \in \CC\setminus \{J\}$ one has $|I|=2$. We are left with proving that $\forall I \in \CC\setminus \{J\}$ we have $I \cap J \neq \varnothing$. Assume $I \cap J= \varnothing$. Since  $I=\{C_1C_2\}$ and $\PP$ is simple thin, then there exists $I_1 \in \CC$ with $|I_1|=2$ such that $I_1=\{C_2 C\}$ with $C \in J$. Since $|J \setminus\{C\}|\geq 2$, then there exists $D_1, D_2 \in J$ such that $\{C_1,D_1,C_2,D_2\}$ is an induced $4$-cycle of $\bar{G}$, that is $\bar{G}$ is not chordal. This leads to a contradiction. It follows that $\PP$ is a brush polyomino.
\end{proof}

From Theorem \ref{fr} one obtains the following 
\begin{Corollary}
Let $\PP$ be a simple polyomino. Then $\reg R/I(G_\PP) \leq 1$ if and only if $\PP$ is a brush polyomino or $\PP$ is one of the polyominoes in Figure \ref{fig:nnthin}.
\end{Corollary}

\section{The Castelnuovo-Mumford regularity of pure brush polyominoes}\label{sec:cmreg}
In this section, we compute the Castelnuovo-Mumford regularity of $R/I(G_\PP)$ for pure brush polyominoes $\PP$. We are motivated by the following observations. 

\begin{Remark}
Let $\PP$ be a simple thin polyomino. We observe that $G_\PP$ contains no induced cycles, hence is chordal and by \cite[Corollary 7.(2)]{Wo} $\RR_\PP$ is vertex decomposable.

\end{Remark}
\begin{Corollary}
Let $\PP$ be a simple thin polyomino. Then the following are equivalent
\begin{itemize}
    \item[(i)] $\RR_\PP$ is pure;
        \item[(ii)]  $\RR_\PP$ is Cohen-Macaulay;
                \item[(iii)] $\RR_\PP$ is vertex decomposable.
\end{itemize}
\end{Corollary}
Therefore, if $\RR_\PP$ is pure and $h(\RR_\PP)=(h_0,\ldots, h_r)$, then $r= \reg R/I(G)$. Hence, for the rest of the section we focus on the class of pure brush polyominoes (see Figure \ref{fig:pbru}). 
\begin{figure}[H]
    \centering
           \begin{tikzpicture}
\draw (0,0)--(5,0);
\draw (0,1)--(5,1);

\draw (0,-2)--(0,1);
\draw (1,-2)--(1,3);
\draw (1,-2)--(0,-2);
\draw (1,-1)--(0,-1);

\draw (1,2)--(2,2);
\draw (1,2)--(2,2);
\draw (1,3)--(2,3);
\draw (2,-1)--(3,-1);
\draw (2,-2)--(3,-2);
\draw (2,-3)--(3,-3);
\draw (4,2)--(3,2);
\draw (4,-1)--(5,-1);

\draw (2,-3)--(2,3);
\draw (3,-3)--(3,1);
\draw (3,0)--(3,2);
\draw (4,-2)--(4,2);
\draw (5,-2)--(5,1);
\draw (5,-2)--(4,-2);
   
    \end{tikzpicture}
    \caption{A pure brush polyomino}\label{fig:pbru}
\end{figure}
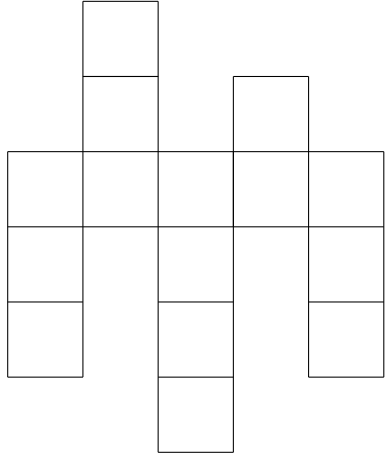

If follows that a pure brush polyomino with $\dim \RR_\PP=d-1$ is such that $\AA=\{I_1,\ldots I_d\}$ with $|I_k|=\ell_k$ for any $k \in \{1,\ldots , d\}$, $\CC=\AA \cup \{J\}$ with $|J|=d$, and $I_k\cap J \neq \varnothing$ for any $k \in \{1,\ldots , d\}$. Let $\ell=(\ell_1,\ldots,\ell_d)$.
We want to study the vectors $f(\RR_\PP)$ and $h(\RR_\PP)$ for a pure brush polyomino $\PP$.

We recall that for $1 \ldots k \ldots d$ the $k$-th elementary symmetric polynomial in $d$ indeterminates $X_1,\ldots X_d$ is 
\[
\epsilon_{k}^{(d)}(X_1,\ldots,X_d)=\sum\limits_{1 \leq i_1 < i_2 < \ldots < i_k\leq d} X_{i_1}X_{i_2}\cdots X_{i_k}.
\]
For any  $1 \ldots k \ldots d$  we set 
\begin{align*}
&\sigma_{k}=\epsilon_{k}(\ell_1,\ldots,\ell_d)\\
&\sigma'_{k}=\epsilon_{k}(\ell_1-1,\ldots,\ell_d-1),\\
&\sigma''_{k}=\epsilon_{k}(\ell_1-2,\ldots,\ell_d-2),\\
\end{align*}

\begin{Lemma}\label{lem:sigma}
For any $d \in \NN$ and $1\leq k \leq d$ the following relations hold
\begin{enumerate}[label=(\roman*)]

    \item $\sigma_{k}'=\sum\limits_{i=0}^{k} (-1)^{k-i} \binom{d-i}{k-i}\sigma_i$
    \item $\sigma_k=\sum\limits_{i=0}^{k} \binom{d-i}{k-i}\sigma'_i$
    \item $\sigma_{k}''=\sum\limits_{i=0}^{k} (-1)^{k-i} \binom{d-i}{k-i}\sigma'_i$ 
\end{enumerate}
\end{Lemma}
\begin{proof}
We consider a subset $\{i_1,\ldots i_k\}\subseteq \{1,\ldots, d\}$ and we consider the number
\begin{equation}\label{eq:ell}
    (\ell_{i_1}-1)(\ell_{i_2}-1)\cdots (\ell_{i_k}-1)
\end{equation}

For any $i \in \{1,\ldots, k\}$ we set 
\[
\sigma_{i}^{(k)}=\epsilon_{i}^{(k)}(\ell_{i_1}, \ldots ,\ell_{i_k}),
\]
hence, from Vieta's formulas, Equation \eqref{eq:ell} becomes
\[
\sum_{i=0}^k (-1)^{k-i} \sigma^{(k)}_{i}
\]
We now consider
\[
\sigma_{k}'=\sum\limits_{1 \leq i_1 < i_2 < \ldots < i_k\leq d} (\ell_{i_1}-1)(\ell_{i_2}-1)\cdots (\ell_{i_k}-1)=\sum\limits_{1 \leq i_1 < i_2 < \ldots < i_k\leq d} \ \ \sum_{i=0}^k (-1)^{k-i} \sigma^{(k)}_{i}
\]
We observe that fixed $\{i_1,\ldots ,i_k\}$ and $A \subseteq \{i_1,\ldots ,i_k\}$ with $|A|=i$, the summand 
\[
\prod\limits_{a \in A} \ell_a
\]
appears $\binom{d-j}{k-j}$ times in $\sigma'_{k}$. That is 
\[
\sigma_{k}'=\sum\limits_{1 \leq i_1 < i_2 < \ldots < i_k\leq d} \ \ \sum_{i=0}^k (-1)^{k-i} \sigma^{(k)}_{i}=\sum_{i=0}^k (-1)^{k-i} \binom{d-j}{k-j} \sigma_{k}
\]
and relation $(i)$ follows. Similarly, relation $(iii)$ follows.

Relation $(ii)$ follows from Relation $(i)$, as a comparision with the relation between Equation \ref{eq:hf} and Equation \ref{eq:fh}.
\end{proof}

Moreover we have 
\begin{Theorem}\label{thm:hfbrush}
Let $\PP$ be a pure brush polyomino with $\dim \RR_\PP=d-1$. Then the following relations hold
\begin{enumerate}
    \item for all $k \in \{1,\ldots d\}$ \[f_{k-1}=\sigma'_{k}+(d-(k-1))\sigma'_{k-1}\]
    \item for all $t \in \{0,\ldots d\}$ \[h_{t}=\sigma''_{t}+(d-(t-1))\sigma''_{t-1}\]
\end{enumerate}
\end{Theorem}
\begin{proof}
We prove relation $(1)$, by first observing that for any $i=2,\ldots , d$ we have
\[
f_{k-1}=\sigma_{k} - \Bigg(\sum\limits_{j=0}^{k-2} \binom{d-j}{k-j} \sigma'_{j} \Bigg).
\]
We have that for any $\{i_1,\ldots , i_k\} \subset \{1,\ldots , d\}$, we have
\[
\ell_{i_1} \cdots \ell_{i_k}
\]
configurations of rooks. From this number for any $0\leq j \leq k-2$ we have to subtract the configurations that have exactly $k-j$ cells on the common interval $J$. Fixed a subset $A=\{a_{1},\ldots a_{k-j} \} \subset \{i_1,\ldots ,i_k\}$, the configurations containing the cells in $J \cap \{I_{a_1}\}, \ldots , J \cap \{I_{a_{k-j}}\} $ are 
\[
\prod\limits_{t \in \{i_1,\ldots , i_k\}\setminus A} (\ell_t - 1)
\]
Let $\SS_{k-j}$ be the set of cardinality $k-j$ subsets of $\{i_1,\ldots , i_k\}$. We consider
\[
\sum\limits_{1 \leq i_1 < i_2 < \ldots < i_k\leq d} \ \  \sum_{A \in \SS_{k-j}} 
\prod\limits_{t \in \{i_1,\ldots , i_k\}\setminus A} (\ell_t - 1).
\]
Since the $j$ elements of $\{i_1,\ldots , i_k\}\setminus A$ are fixed, then in the above sum any product $\binom{d-j}{k-j}$ is counted times, and their sum retrieves $\sigma'_{j}$.
Hence,

\[
f_{k-1}= \sum\limits_{1 \leq i_1 < i_2 < \ldots < i_k\leq d} \ell_{i_1} \cdots \ell_{i_k} - \Bigg( \sum_{j=0}^{k-2} \sum_{A \in \SS_{k-j}} 
\prod\limits_{t \in \{i_1,\ldots , i_k\}\setminus A} (\ell_t - 1).
\Bigg)=
\]
\[
=\sum\limits_{1 \leq i_1 < i_2 < \ldots < i_k\leq d} \ell_{i_1} \cdots \ell_{i_k} - \Bigg( \sum_{j=0}^{k-2} \ \  \sum\limits_{1 \leq i_1 < i_2 < \ldots < i_k\leq d} \ \  \sum_{A \in \SS_{k-j}} 
\prod\limits_{t \in \{i_1,\ldots , i_k\}\setminus A} (\ell_t - 1).
\Bigg)=
\]
\[
=\sigma_{k}- \sum_{j=0}^{k-2} \binom{d-j}{k-j} \sigma'_{j}.
\]
According to relation $(ii)$ of Lemma \ref{lem:sigma}, we have that 
\[
f_{k-1}=\sum\limits_{j=k-1}^k  \binom{d-j}{k-j} \sigma'_{j}= \sigma'_{k}+(d-(k-1))\sigma'_{k-1}.
\]

To prove relation (2), we consider Equation \eqref{eq:hf}, that is
\[
h_{t}=\sum_{k=0}^t (-1)^{t-k} \binom{d-k}{t-k} f_{k-1}=
\]
\[
=(-1)^{t} \binom{d}{t} f_{-1} + \sum_{k=1}^t (-1)^{t-k} \binom{d-k}{t-k} f_{k-1}
\]
By using relation $(1)$, we obtain 
\[
=(-1)^{t} \binom{d}{t} f_{-1} + \sum_{k=1}^t (-1)^{t-k} \binom{d-k}{t-k} (\sigma'_{k}+(d-(k-1))\sigma'_{k-1})=
\]
one observes that $f_{-1}=\sigma'_{0}=1$
\[
=\sum_{k=0}^t (-1)^{t-k} \binom{d-k}{t-k}\sigma'_{k} + \sum_{k=1}^t (-1)^{t-k}\binom{d-k}{t-k}(d-k+1)\sigma'_{k-1}= (*)
\]
We observe that 
\[
(d-k+1)\binom{d-k}{t-k}= (d-k+1)\frac{(d-k)!}{(t-k)!(d-t)!} \cdot \frac{d-t+1}{d-t+1}= (d-t+1)\binom{d-k+1}{t-k}
\]
that is 
\[
(*)=\sum_{k=0}^t (-1)^{t-k} \binom{d-k}{t-k}\sigma'_{k} +(d-t+1)\sum_{k=1}^t (-1)^{t-k} \binom{d-k+1}{t-k} \sigma'_{k-1}=
\]
In the second sum, we substitue $j=k-1$ and we obtain
\[
=\sum_{k=0}^t (-1)^{t-k} \binom{d-k}{t-k}\sigma'_{k} +(d-t+1) \sum_{j=0}^{t-1} (-1)^{t-j+1} \binom{d-j}{t-1-j} \sigma'_{j}.
\]
By using relation (iii) of Lemma \ref{lem:sigma}, we obtain that 
\[
h_{t}=\sigma''_{t}+(d-t+1)\sigma''_{t-1},
\]
as desired.
\end{proof}

We translate the definition of induced matching number for the graph $G$ in terms of the intervals of $\PP$.
\begin{Definition}\label{def:matp}
An induced matching is a set of edges  $\{\{A_1,B_1\},\ldots , \{A_n,B_n\}\}$ such that for any $j,k \in \{1,\ldots, n\}$, if $\{A_j,B_j\} \subset I_j \in \CC$ and $\{A_k,B_k\} \subset I_k \in \CC$, then there is no $J \in \CC$ such that 
\[
J \cap I_j \subset \{A_j,B_j\} \mbox{ and } J \cap I_k \subset \{A_k,B_k\}.
\]
\end{Definition}

For a polyomino $\PP$, we set $\SS= \{I \in \CC : I \mbox{ has at least 2 single cells} \}$

\begin{Lemma}\label{lem:match}
Let $\PP$ be a simple polyomino. 
Then
\[
\nu(G) \geq |\SS|.
\]
\end{Lemma}
\begin{proof}
Let $\SS=\{I_1,\ldots, I_m\}$. For any $j \in \{1,\ldots , m\}$, let $A_j$ and $B_j$ be two single cells of $I_j$. It follows that 
\[
\{\{A_1,B_1\},\ldots , \{A_m,B_m\}\}
\]
is an induced matching, hence $\nu(G) \geq |\SS|$.
\end{proof}

\begin{Corollary}
Let $\PP$ be a pure brush polyomino. Then
\[
\reg R/I(G)= \nu (G).
\]
\end{Corollary}
\begin{proof}
Let $\PP$ be a pure brush polyomino with $\CC=\{J,I_1,\ldots , I_d\} $. We observe that if $l_k \geq 3$ then $I_k$ has two single cell. We distinguish two cases
\begin{enumerate}
    \item for any $k \in \{1,\ldots, d\}$ $\ell_k \geq 3$; 
    \item there exists $j \in \{1,\ldots, d\}$ such that $\ell_j = 2$; 
\end{enumerate}
In case (1), from Lemma \ref{lem:match} we have that $d=\nu(G)$, that is from Theorem \ref{kat}
\[
\nu(G) \leq \reg R/I(G) \leq d,
\]
and the assertion follows.\\
In case (2), we relabel the intervals $I_1,\ldots I_d$ in a way such that 
\[
\ell_{1},\ldots , \ell_{t} > 3 \mbox{ and } \ell_{t+1}=\ldots = \ell_{d}=2
\]
for $t<d$. In this case we have $\nu(G)=t+1$. In fact, since $\SS={I_1,\ldots , I_t}$, then we have $t$ edges in an induced matching. To this we add the unique edge arising from $I_{t+1}$. Combining the facts $I_{k} \cap J\neq \varnothing$ for any $k \in \{1,\ldots, d\}$ and  $\ell_{t+1}=\ldots = \ell_{d}=2$, we obtain that the intervals $I_1,\ldots,I_{t+1}$ give rise to an induced matching. We prove that $h_k = 0$ for $k>t+1$.

Let $k>t+1$. In any cardinality-$k$ subset of $\ell_1-2, \ldots, \ell_d-2$ there is a $0$, that is
\[
\sigma''_{k}=\sigma''_{k-1}=0.
\]
From relation (2) of Theorem \ref{thm:hfbrush}, we obtain $h_k=0$ as desired.
\end{proof}

As a conclusion, in this paper we characterize the polyominoes having a pure rook complex. It could be of interest finding a characterization for those polyominoes having a Cohen-Macaulay (shellable, vertex decomposable) rook complex, and among them finding the Gorenstein ones.

\end{document}